\documentclass[11pt]{amsproc}
 \usepackage[margin=1in]{geometry}
\usepackage{setspace,fullpage}
\usepackage{graphicx}
\usepackage[nice]{nicefrac}
\usepackage{amssymb}
\usepackage{amsmath,amsthm,amscd,amssymb, mathrsfs}

\usepackage{latexsym}
\usepackage[colorlinks,citecolor=red,pagebackref,hypertexnames=false]{hyperref}

\numberwithin{equation}{section}

\theoremstyle{plain}
\newtheorem{theorem}{Theorem}[section]
\newtheorem{lemma}[theorem]{Lemma}
\newtheorem{corollary}[theorem]{Corollary}

\theoremstyle{definition}

\theoremstyle{remark}
\newtheorem{remark}[theorem]{Remark}

\newtheorem{case[theorem]}{Case}

\title[\parbox{14cm}{\centering{The Erd\"{o}s-Falconer distance problems  \hspace{1in}}} \quad]{ The generalized Erd\"{o}s-Falconer distance problems in vector spaces over finite fields }
\author{ Doowon Koh and Chun-Yen Shen }

\address{Department of Mathematics\\
Michigan State University \\
East Lansing, MI 48824,  USA}
\email{koh@math.msu.edu}

\address{Department of Mathematics\\
 Indiana University\\
Bloomington, IN 47405, USA}
\email{shenc@umail.iu.edu}



\begin{document}

\begin{abstract} In this paper we study the generalized Erd\"{o}s-Falconer distance problems in the finite field setting.
The generalized distances are defined in terms of polynomials, and various formulas for sizes of distance sets are obtained.
In particular, we develop a simple formula for estimating the cardinality of distance sets determined by diagonal polynomials.
As a result, we generalize the spherical distance problems due to Iosevich and Rudnev \cite{IR07} and the cubic distance problems due to Iosevich and Koh \cite{IK08}. Moreover, our results are of higher dimensional version for Vu's work \cite{Vu08} on two dimension.
In addition, we set up and study the generalized pinned distance problems in finite fields.
We give a generalization of the work by the authors \cite{CEHIK09} who studied the pinned distance problems related to spherical distances.
Discrete Fourier analysis and exponential sum estimates play an important role in our proof.

\end{abstract}
\maketitle
\tableofcontents

\section{Introduction}

The Erd\H os distance problem, in a generalized sense, is a question of how many distances are determined by a set of points. This problem might be the most well-known problem in discrete geometry. One may consider discrete, continuous and finite field formulations of this question.
Given  finite subsets $E , F$ of ${\bf R^d}$, $d \ge 2,$ the distance set determined by the sets $E, F$ is defined by $ \Delta(E,F)=\{|x-y|: x\in E, y\in F\},$ where $|x|=\sqrt{x_1^2+\dots+x_d^2}$. In the case when $E=F$,  Erd\H os \cite{Er46} asked us to determine the smallest possible size of $\Delta(E,E)$ in terms of the size of $E$. This problem is called the Erd\H os distance problem and it has been conjectured that
$ |\Delta(E,E)| \gtrapprox {|E|}^{2/d}$ where $|\cdot|$ denotes the cardinality of the finite set. Taking $E$ as a piece of the integer lattice shows that one can not in general get the better exponent than $2/d$ for the conjecture. For all dimensions $d\geq 2,$ this problem has not been solved. In two dimension, the best known result is the work by Katz and Tardos \cite{KT04}, which is based on a previous breakthrough by Solymosi and T\'{o}th \cite{SolTo01}. For the best known results in higher dimensions see \cite{SV04} and \cite{SV05}. These results are a culmination of efforts going back to the paper by Erd\H os \cite{Er46}. \\

On the other hand, one can also study  the continuous analog of the Erd\H os distance problem, called the Falconer distance problem. This problem is to determine the Hausdorff dimension of compact sets such that the Lebesque measure of the distance sets is positive. Let $E \subset {\bf R^d}, d\geq 2,$ be a compact set. The Falconer distance conjecture says that  if dim$(E)> d/2,$ then $|\Delta(E,E)| >0$, where dim$(E)$ denotes the Hausdorff dimension of the set $E,$ and $|\Delta(E,E)|$ denotes one dimensional Lebesque measure of the distance set $\Delta(E,E)=\{|x-y|: x,y \in E\}.$
Using the Fourier transform method, Falconer \cite{Fa85} proved that if dim$(E) > (d+1)/2,$ then $|\Delta(E,E)|>0.$ This result was generalized by Mattila \cite{Ma87} who showed that
$$\mbox{if}\quad \mbox{dim}(E)+\mbox{dim}(F)> d+1, ~~\mbox{then}~~ |\Delta(E,F)|>0,$$
where $E, F$ are compact subsets of ${\bf R^d}$ and $\Delta(E,F)=\{|x-y|\in {\bf R}: x\in E, y\in F\}.$
In particular, he made a remarkable observation that the Falconer distance problem is closely related to estimating the upper bound of the spherical means of Fourier transforms of measures. Using the Mattila's method, Wolff \cite{Wo99} obtained the best known result on the Falconer distance problem in two dimension. He proved that if dim$(E)> 4/3$, then $|\Delta(E,E)|>0.$  The best known results for higher dimensions are due to
Erdo\~{g}an \cite{Er05}. Applying the Mattila's method and the weighted version of Tao's bilinear extension theorem \cite{Ta03}, he proved that
if $\mbox{dim}(E)> d/2+1/3$, then $|\Delta(E,E)| >0,$ where $d\geq 2$ is the dimension. However, the Falconer distance problem is still open for all dimensions $d\geq 2.$ As a variation of the Falconer distance problem, Peres and Schlag \cite{PS00} studied the pinned distance problems and showed that the Falconer result can be sharpen. More precisely,
they proved that if $ E\subset {\bf R^d}$ and dim$(E)> (d+1)/2$, then $ |\Delta(E, y)|>0$ for almost every $y\in E,$ where
the pinned distance set $ \Delta(E,y)$ is given by
$$ \Delta(E,y)=\{|x-y|: x\in E\}.$$\\

In recent years the Erd\H os-Falconer distance problem has been also studied in the finite field setting.
Let $\mathbb F_q$ be a finite field with $q$ elements. We denote by $\mathbb F_q^d, d\geq 2$, the $d$-dimensional vector space over the finite field ${\mathbb F}_q.$
Given a polynomial $P(x)\in \mathbb F_q[x_1,\dots, x_d]$ and $E, F \subset \mathbb F_q^d,$ one may define a generalized distance set $\Delta_P(E,F)$ by the set
\begin{equation}\label{defgd} \Delta_P(E, F) =\{ P(x-y)\in \mathbb F_q: x\in E, y\in F\}.\end{equation}
In the case when $E=F$ and $P(x)=x_1^2+x_2^2,$  Bourgain, Katz and Tao \cite{BKT04}  first obtained the following nontrivial result on the Erd\H os distance problem in the finite field setting: if $q$ is prime with $q \equiv 3~ (\mbox{mod}~4)$ and $E \subset \mathbb F_q^2$ with  $|E|=q^\delta$ for some $0<\delta<2,$ then there exists $\varepsilon=\varepsilon(\delta)>0$ such that
\begin{equation}\label{taoerdos}|\Delta_P(E,E)| \gtrsim |E|^{\frac{1}{2}+\varepsilon},\end{equation}
where we recall that if $A, B$ are positive numbers, then $A\lesssim B$ means that there exists $C>0$ independent of $q$, the cardinality of the underlying finite field $\mathbb F_q$ such that $A\leq CB.$ However, if there exists $i\in \mathbb F_q$ with $i^2=-1,$  or the field $\mathbb F_q$ is not the prime field, then the inequality (\ref{taoerdos}) can not be true in general. For example, if we take $E=\{(s, is)\in \mathbb F_q^2: s\in \mathbb F_q\},$ then $|E|=q$ but $|\Delta_P(E,E)|=|\{0\}|=1.$ Moreover, if $q=p^2$ with  $p$ prime, and  $E=\mathbb F_p^2,$ then $|E|=p^2=q$ but
$ |\Delta_P(E,E)|=p=\sqrt{q}.$ In view of these examples, Iosevich and Rudnev \cite{IR07} replaced the question on the Erd\H os distance problems by
the following Falconer distance problem in the finite field setting: how large a set $E\subset \mathbb F_q^d$ is needed to obtain a positive proportion of all distances. They first showed that if $|E|\geq 2q^{(d+1)/2}$ then one can obtain all distances that is $|\Delta_P(E,E)|=q$ where $P(x)=x_1^2+\cdots+x_d^2.$
In addition, they conjectured that $|E|\gtrsim q^{\frac{d}{2}}$ implies that $|\Delta_P(E,E)|\gtrsim q.$
In the case when $P(x)=x_1^s+\cdots+x_d^s, s\geq 2,$ more general conjecture was given by Iosevich and Koh \cite{IK08}.
However, it turned out that in the case $s=2$ if one wants to obtain all distances, then arithmetic examples constructed by authors in \cite{HIKR10} show that the exponent $(d+1)/2$ is sharp in odd dimensions. The problems in even dimensions are still open. Moreover if one wants to obtain a positive proportion of all distances, then the exponent $(d+1)/2$ was recently improved in two dimension by the authors in \cite{CEHIK09} who proved that if $E\subset \mathbb F_q^2$ with $|E|\gtrsim q^{4/3},$ then $|\Delta_P(E,E)|\gtrsim q$  where $P(x)=x_1^2+x_2^2.$ This result was generalized by Koh and Shen \cite{KS10} in the sense that if $E,F \subset \mathbb F_q^2$ and $|E||F|\gtrsim q^{8/3},$ then $|\Delta_P(E,F)|=|\{P(x-y)\in \mathbb F_q: x\in E, y\in F\}| \gtrsim q.$\\

In this paper, we shall study  the Erd\H os-Falconer distance problems for finite fields, associated with the generalized distance set defined as in (\ref{defgd}). This problem can be considered as a generalization of the spherical distance problems and the cubic distance problems which were studied by Iosevich and Rudnev in \cite{IR07} and Iosevich and Koh in \cite{IK08} respectively. The generalized Erd\H os distance problem  was first introduced by Vu \cite{Vu08}, mainly studying the size of the distance sets, generated by non-degenerate polynomials $P(x)\in \mathbb F_q[x_1,x_2].$ Using the spectral graph theory, he proved that if $P(x)\in \mathbb F_q[x_1,x_2]$ is a non-degenerate polynomial and $E \subset \mathbb F_q^2 $ with $|E|\gtrsim q,$  then we have
\begin{equation}\label{byvu}|\Delta_P(E,E)|\gtrsim \min \left( q, |E| q^{-\frac{1}{2}} \right)\end{equation}
where a polynomial $P(x)\in \mathbb F_q[x_1,x_2]$ is called a non-degenerate polynomial if it is not of the form $G(L(x_1,x_2))$ where $G$ is an one-variable polynomial and $L$ is a linear form in $x_1,x_2.$ In order to obtain the inequality (\ref{byvu}), the assumption $|E|\gtrsim q$ is  necessary in general setting, which is clear from the following example:
if $P(x)=x_1^2-x_2^2$ and $E=\{(t,t)\in \mathbb F_q^2: t\in \mathbb F_q \}$ is the line, then we see that
$|E|=q$ and $|\Delta_P(E,E)|=|\{0\}|=1$ and so the inequality (\ref{byvu}) can not be true. Using the Fourier analysis method,  Hart, Li, and Shen \cite{HLS10} showed that  $P(x)-b \in \mathbb F_q[x_1,x_2]$ does not have any linear factor for all $b\in \mathbb F_q$ if and only if  the following inequality holds:
\begin{equation}\label{byshen}
|\Delta_P(E,F)|\gtrsim \min \left( q, \sqrt{|E||F|} q^{-\frac{1}{2}}\right) \quad \mbox{for all}~~ E, F\subset \mathbb F_q^2.
\end{equation}

In the finite field setting,  results on the Erd\H os distance problem implies  results on the Falconer distance problem.
For example, the inequality (\ref{byshen}) implies that if $E, F\subset \mathbb F_q^2$ with $|E||F|\gtrsim q^3$, then $\Delta_P(E,F)$ contains a positive proportion of all possible distances, that is $|\Delta_P(E,F)|\gtrsim q.$ \\

The purpose of this paper is to develop the two-dimensional work  by Vu \cite{Vu08} to higher dimensions. In terms of the Fourier decay on varieties generated by general polynomials, we classify the size of distance sets. In particular, we investigate the size of the generalized Erd\H os-Falconer distance sets related to diagonal polynomials, that are of the form
$$P(x)= \sum_{j=1}^d a_jx_j^{c_j} \in \mathbb F_q[x_1,\dots,x_d]$$
where $ a_j\ne 0$ and $c_j\geq 2$  for all $i=1,\dots,d.$
The polynomial $P(x)=\sum_{j=1}^d x_j^2$ is related to the spherical distance problem. In this case, the Erd\H os-Falconer distance problems were well studied by Iosevich and Rudnev \cite{IR07}. On the other hand, Iosevich and Koh \cite{IK08} studied the cubic distance problems associated with the polynomial $P(x)=\sum_{j=1}^d x_j^3.$ In addition,  Vu's theorem (\ref{byvu}) gives us some results on the Erd\H os-Falconer distance problems in two dimension related to the polynomial $P(x)= a_1x_1^{c_1}+a_2x_2^{c_2}.$ As we shall see, our results will recover and extend the aforementioned authors' work. We also study the generalized pinned distance problems in the finite field setting.
As the analogue of the Euclidean pinned distance problem,  the authors in \cite{CEHIK09} considered the following pinned distance set:
$$\Delta_P(E,y)=\{P(x-y)\in \mathbb F_q: x\in E\}$$
where  $E\subset \mathbb F_q^d, y\in \mathbb F_q^d$ and $P(x)=x_1^2+\cdots+x_d^2.$ Using the fact that for $x,x^\prime, y \in \mathbb F_q^d,$
\begin{equation}\label{same} P(x-y)-P(x^\prime-y) =(P(x)-2 y\cdot x)-(P(x^\prime)-2y\cdot x^\prime ),\end{equation}
they obtained the following strong result.
\begin{theorem}\label{pinmythm}
Let $E\subset \mathbb F_q^d, d\geq 2.$ If $|E|\geq q^{\frac{d+1}{2}},$ then there exists $E^\prime \subset E$ with $|E^\prime|\sim |E|$ such that
$$ |\Delta_P(E, y)| > \frac{q}{2} \quad \mbox{for all} ~~ y\in E^\prime,$$
where $P(x)=x_1^2+\cdots+x_d^2.$ \end{theorem}
However, if the polynomial $P(x)$ is replaced by a general polynomial in $\mathbb F_q[x_1,\dots,x_d],$ then the  equality (\ref{same}) can not be in general obtained. Thus, the main idea for the proof of Theorem \ref{pinmythm} could not be applied to the generalized pin distance problems. Investigating the Fourier decay on the variety generated by a general polynomial, we shall generalize Theorem \ref{pinmythm}. For instance, our result implies that  such fact as above theorem can be obtained if the polynomial $P$ is a diagonal polynomial with all exponents same.

\section{Discrete Fourier analysis and exponential sums}
In order to prove our main results on the generalized Erd\H os-Falconer distance problems, the discrete Fourier analysis shall be used as the principle tool. In this section, we review the discrete Fourier analysis machinery for finite fields, and collect well-known facts on classical exponential sums.
\subsection{ Finite Fourier analysis}
Let $\mathbb F_q^d, d\geq 2,$ be a $d$-dimensional vector space over the finite field $\mathbb F_q$ with $q$ element. We shall work on the vector space $\mathbb F_q^d,$ and throughout the paper, we shall assume that the characteristic of the finite field $\mathbb F_q$ is sufficiently large so that some minor technical problems can be overcome.
We denote by $\chi: \mathbb F_q \rightarrow \mathbb S^1$ the canonical additive character of $\mathbb F_q.$
For example, if $q$ is prime, then we can take $ \chi(s)=e^{2\pi i s/ q}.$ For the example of the canonical additive character of the general field $\mathbb F_q,$ see Chapter $5$ in \cite{LN93}. Let $f: \mathbb F_q^d \rightarrow \mathbb C$ be a complex valued function on $\mathbb F_q^d.$ Then, the Fourier transform of the function $f$ is defined by
\begin{equation}\label{Ftransform} \widehat{f}(m)= \frac{1}{q^d} \sum_{x\in \mathbb F_q^d} f(x) \chi(-x\cdot m)\quad \mbox{for}~~m\in \mathbb F_q^d.\end{equation}
We also recall in this setting that the Fourier inversion theorem says that
\begin{equation}\label{Finversion} f(x)=\sum_{m\in \mathbb F_q^d} \chi(x\cdot m) \widehat{f}(m).\end{equation}
Using the orthogonality relation of the canonical additive character $\chi$, that is $ \sum_{x\in \mathbb F_q^d} \chi(x\cdot m)=0 $ for $m\neq (0,\dots,0)$ and $ \sum_{x\in \mathbb F_q^d} \chi(x\cdot m)= q^d $ for $m=(0,\dots,0)$,  we obtain the following Plancherel theorem:
$$ \sum_{m\in \mathbb F_q^d} |\widehat{f}(m)|^2 = \frac{1}{q^d} \sum_{x\in \mathbb F_q^d} |f(x)|^2.$$
For example, if $f$ is a characteristic function on the subset $E$ of $\mathbb F_q^d,$ then we see
\begin{equation}\label{Plancherel} \sum_{m\in \mathbb F_q^d} |\widehat{E}(m)|^2 = \frac{|E|}{q^d},\end{equation}
here, and throughout the paper, we identify the set $E \subset \mathbb F_q^d$ with the characteristic function on the set $E$, and we denotes by $|E|$ the cardinality of the set $E \subset \mathbb F_q^d.$

\subsection{ Exponential sums} Using the discrete Fourier analysis, we shall make an effort to reduce the generalized Erd\H os-Falconer distance problems to estimating classical exponential sums. Some of our formulas for the distance problems can be directly applied via recent well-known
exponential sum estimates. For example, the following lemma is well known and it was obtained by applying cohomological arguments (see Example 4.4.19 in \cite{Co94}).
\begin{lemma}\label{Todd} Let $P(x)=\sum\limits_{j=1}^d a_jx_j^s \in \mathbb F_q[x_1,\dots,x_d]$ with $s\geq 2, a_j\neq 0$ for all $j=1,\dots,d.$
In addition, assume that the characteristic of $\mathbb F_q$ is sufficiently large  so that  it does not divide $s.$ Then,
$$ |\widehat{V_t}(m)|=\frac{1}{q^d} \left|\sum_{x\in V_t} \chi(-x\cdot m)\right| \lesssim q^{-\frac{d+1}{2}} \quad \mbox{for all}~~ m\in \mathbb F_q^d\setminus \{(0,\dots,0)\}, t\in \mathbb F_q\setminus\{0\},$$
and $$ |\widehat{V_0}(m)| \lesssim q^{-\frac{d}{2}} \quad \mbox{for all} ~~ m \in \mathbb F_q^d\setminus \{(0,\dots,0)\},$$
where $V_t =\{x\in \mathbb F_q^d: P(x)=t\}.$
\end{lemma}
However, some theorems obtained by cohomological arguments contain abstract assumptions, and it can be often hard to apply them in practice. In order to overcome this problem, we shall also develop an alternative formula which is closely related to more simple exponential sums. As we shall see, such a simple formula can be obtained by viewing the distance problem in $d$ dimensions as the distance  problem for product sets in $(d+1)-$dimensional vector spaces. As a typical application of our simple distance formula, we shall obtain the results on the Falconer distance problems related to arbitrary diagonal polynomials, which take the following forms: $P(x)=\sum_{j=1}^d a_jx_j^{c_j}$ for $c_j\geq 2, a_j\neq 0$ for all $j.$ It is shown that such results can be obtained by applying  the following well-known Weil's theorem. For a nice proof of Weil's theorem, we refer readers to Theorem 5.38 in \cite{LN93}.
\begin{theorem}\label{Weil} [Weil's Theorem]
Let $f\in \mathbb F_q[s]$ be of degree $c\geq 1$ with gcd$(c,q)=1$. Then, we have
$$ \left|\sum_{s\in \mathbb F_q} \chi(f(s))\right| \leq (c-1) q^{\frac{1}{2}},$$
where $\chi$ denotes a nontrivial additive character of $\mathbb F_q.$
\end{theorem}

We now collect well-known facts which make a crucial role in the proof of our main results.
First, we introduce the cardinality of varieties related to arbitrary diagonal polynomials.
The following theorem is due to Weil \cite{We49}. See also Theorem 3.35 in \cite{Co94} or Theorem 6.34 in \cite{LN93}.
\begin{theorem}\label{Cardinality} Let $P(x)=\sum_{j=1}^d a_jx_j^{c_j}$ with $a_j\neq 0, c_j\geq 1$ for all $j=1,\dots,d.$
For every $t\in\mathbb F_q \setminus \{0\},$ we have
$$ |\{x\in \mathbb F_q^d: P(x)=t\}| \sim q^{d-1}.$$\end{theorem}

The following lemma is known as the Schwartz-Zippel Lemma (see \cite{Zi79} and \cite{Sc80}).
A nice proof is also given in Theorem $6.13$ in \cite{LN93}.
\begin{lemma}\label{SZlemma}[Schwartz-Zippel] Let $P(x)\in \mathbb F_q[x_1,\dots, x_d] $ be a non zero polynomial with
degree $k$. Then, we have
$$ |\{x\in \mathbb F_q^d: P(x)=0\}|\leq kq^{d-1}.$$
\end{lemma}

We also need the following theorem which was implicitly given in \cite{Vu08}.
\begin{theorem}\label{Vuthm} Let $P(x)\in \mathbb F_q[x_1,x_2] $ be a non-degenerate polynomial of degree $k\geq 2.$
Then there is a set $T\subset \mathbb F_q$ with $0\leq |T|\leq (k-1),$  such that for every $m\in \mathbb F_q^2\setminus \{(0,0)\}, t\notin T,$
$$ |\widehat{V_t}(m)|= \frac{1}{q^2} \sum_{x\in  V_t} \chi(-x\cdot m) |\lesssim q^{-\frac{3}{2}},$$
where $ V_t=\{x\in \mathbb F_q^2: P(x)=t\}$ for $t\in \mathbb F_q.$
\end{theorem}
\begin{remark}\label{doit} In Theorem \ref{Vuthm}, it is clear that if $t\in T$, then
\begin{equation}\label{Jfinal} |\widehat{V_t}(m)|\lesssim q^{-1} \quad \mbox{for all}~~m\in \mathbb F_q^2.
\end{equation}
This follows immediately from the Schwartz-Zippel lemma and the simple observation that $|\widehat{V_t}(m)|\leq q^{-2}|V_t|.$\end{remark}

\section{Distance formulas based on the Fourier decays}
Following the similar skills due to Iosevich and Rudnev \cite{IR07}, we shall obtain  the generalized distance formulas. As an application of  the formulas, we will obtain results on the generalized Erd\H os-Falconer distance problems associated with specific diagonal polynomials $P(x)=\sum_{j=1}^d a_jx_j^s.$  Let $P(x)\in \mathbb F_q[x_1,\dots,x_d]$ be a polynomial with degree $\geq 2.$
Given sets $E, F\subset \mathbb F_q^d,$ recall that a generalized pair-wise distance set $\Delta_P(E,F)$ is given by the set
$$\Delta_P(E,F)=\{P(x-y)\in \mathbb F_q: x\in E, y\in F\}.$$
For the Erd\H os distance problems, we aim to find the lower bound of $|\Delta_P(E,F)|$ in terms of $|E|, |F|.$
For the Falconer distance problems, our goal is to determine an optimal exponent $s_0 >0$ such that if $|E||F|\gtrsim q^{s_0}$, then $|\Delta_P(E,F)| \gtrsim q.$
In this general setting, the main difficulty on these problems is that we do not know the explicit form of the polynomial $P(x)\in \mathbb F_q[x_1,\dots,x_d],$  generating  generalized distances.
Thus, we first try to find some conditions on the variety $V_t=\{x\in \mathbb F_q^d: P(x)=t\}$ for $t\in \mathbb F_q$ such that some results can be obtained for the distance problems. In view of this idea, we have the following distance formula.
\begin{theorem}\label{alldistanceformula}
Let $E, F \subset {\mathbb F_q^d}$ and $P(x)\in {\mathbb F_q}[x_1,\dots,x_d].$
For each $t\in {\mathbb F_q}$, we let
\begin{equation}\label{variety} V_t=\{x\in {\mathbb F_q^d}: P(x)-t=0\}.\end{equation}
Suppose that there is a set $T\subset {\mathbb F_q}$ such that $|V_t|\sim q^{d-1}$ for all $ t\in {\mathbb F_q}\setminus T$ and
\begin{equation}\label{SharpDecay} \left|\widehat{V_t}(m)\right| \lesssim q^{-\frac{d+1}{2}}\quad \mbox{for all} \quad t\notin T , m\in {\mathbb F_q^d}\setminus \{(0,\dots,0)\}.\end{equation}
Then, if $ |E||F|\geq C q^{d+1}$ with $C>0$ sufficiently large, we have
$$\left|\Delta_P(E,F)\right| \geq q-|T|.$$
\end{theorem}
\begin{proof}
Consider the counting function $\nu$ on ${\mathbb F_q}$ given by
$$ \nu(t)=\left| \{(x,y)\in E\times F: P(x-y)=t\}\right|.$$
It suffices to show that  $\nu(t)\neq 0$ for every $t\in \mathbb F_q \setminus T.$
Fix $t\notin T.$
Applying the Fourier inversion theorem (\ref{Finversion}) to $ V_t(x-y)$ and using the definition of the Fourier transform (\ref{Ftransform}), we have
$$\nu(t)= \sum_{x\in E, y\in F} V_t(x-y)= q^{2d} \sum_{m\in {\mathbb F_q^d}}\overline{\widehat{E}}(m) \widehat{F}(m) \widehat{V_t}(m).$$
Write $\nu(t)$ by
\begin{align}\label{counting}
\nu(t)=&q^{2d} \overline{\widehat{E}}(0,\dots,0) \widehat{F}(0,\dots,0) \widehat{V_t}(0,\dots,0) +
q^{2d} \sum_{m\in {\mathbb F_q^d}\setminus \{(0,\dots,0)\}}\overline{\widehat{E}}(m) \widehat{F}(m) \widehat{V_t}(m)\\
=&\mbox{I} + \mbox{II}\nonumber.
\end{align}
From the definition of the Fourier transform, we see
\begin{equation}\label{finalI} 0< \mbox{I}= \frac{1}{q^d} |E||F||V_t|.\end{equation}
On the other hand, the estimate $(\ref{SharpDecay})$ and the Cauchy-Schwarz inequality yield
$$|\mbox{II}|\lesssim q^{2d}q^{-\frac{d+1}{2}} \left(\sum_{m} \left|\overline{\widehat{E}}(m)\right|^2\right)^\frac{1}{2}
\left(\sum_{m} \left|\widehat{F}(m)\right|^2\right)^\frac{1}{2}.$$
Applying the Plancherel theorem (\ref{Plancherel}), we obtain
\begin{equation}\label{finalII}
|\mbox{II}| \lesssim q^{\frac{d-1}{2}} |E|^{\frac{1}{2}} |F|^{\frac{1}{2}}.
\end{equation}
Since $|V_t|\sim q^{d-1}$ for each $t\in {\mathbb F_q}\setminus T$,
comparing (\ref{finalI}) with (\ref{finalII}) gives the complete proof.
\end{proof}

As a generalized version of spherical distance problems in \cite{IR07} and cubic distance problems in \cite{IK08}, we have the following corollary.
\begin{corollary} \label{Cor1}Let $P(x)=\sum_{j=1}^d a_j x_j^s \in \mathbb F_q[x_1,\dots,x_d]$ for $s\geq 2$ integer and $a_j\neq 0.$
Suppose that the characteristic of $\mathbb F_q$ is sufficiently large.
If $|E||F|\geq C q^{d+1}$ for $E,F\subset \mathbb F_q^d,$ then $|\Delta_P(E,F)|=q-1,$
where $C>0$ is a sufficiently large constant.
\end{corollary}
\begin{proof} The statement in Corollary \ref{Cor1}  follows immediately from Theorem \ref{alldistanceformula} along with Lemma \ref{Todd} and Theorem \ref{Cardinality}. \end{proof}
Under the assumptions in Corollary \ref{Cor1}, we do not know whether the distance set $\Delta_P(E,F)$ contains zero or not. However, if $E=F$, then $0\in \Delta_P(E,F).$ In this case, the distance set contains all possible distances.\\

Theorem \ref{alldistanceformula} may provide us of an exact size of distance set $\Delta_P(E,F)$ and it may be a useful theorem for the Falconer distance problems for finite fields. However, if $|E||F|$ is much smaller than $q^{d+1}$, then Theorem \ref{alldistanceformula} does not give any information about the size of the distance set $\Delta_P(E,F).$ Now, we introduce another generalized distance formula which is useful for the Erd\H os distance problems in the finite field setting.
\begin{theorem}\label{Erdosformula}Let $E, F \subset {\mathbb F_q^d}$ and $P(x)\in {\mathbb F_q}[x_1,\dots,x_d].$
For each $t\in \mathbb F_q,$ the variety $V_t$ is defined as in (\ref{variety}).
Suppose that there exists a set $A\subset \mathbb F_q$ with $|A|\sim 1$ such that
\begin{equation}\label{hypothesis1} |\widehat{V_t}(m)| \lesssim q^{-\frac{d+1}{2}} \quad \mbox{for all}~~t\notin A, m\in \mathbb F_q^d\setminus \{(0,\dots,0)\}\end{equation}
and
\begin{equation}\label{hypothesis2}|\widehat{V_t}(m)| \lesssim q^{-\frac{d}{2}} \quad \mbox{for all}~~t\in A, m\in \mathbb F_q^d\setminus \{(0,\dots,0)\}.\end{equation}
If $|E||F|\geq C q^{d}$ for some $C>0$ sufficiently large, then we have
$$ |\Delta_P(E,F)| \gtrsim \min\left( q, q^{-\frac{(d-1)}{2}} \sqrt{|E||F|}\right).$$
\end{theorem}
\begin{proof} From (\ref{counting}) and (\ref{finalI}), we see that for every $t\in \mathbb F_q,$
$$ \nu(t)=\left| \{(x,y)\in E\times F: P(x-y)=t\}\right|$$
$$=\frac{1}{q^d} |E||F||V_t| +q^{2d} \sum_{m\in {\mathbb F_q^d}\setminus \{(0,\dots,0)\}}\overline{\widehat{E}}(m) \widehat{F}(m) \widehat{V_t}(m).$$
$$\lesssim \frac{|E||F|}{q} +q^{2d} \left( \max_{m\neq (0,\dots,0)} |\widehat{V_t}(m)|\right) \sum_{m\in \mathbb F_q^d}|\overline{\widehat{E}}(m)| |\widehat{F}(m)|,$$
where we also used the Schwartz-Zippel lemma (Theorem \ref{SZlemma}).
From the Cauchy-Schwarz inequality and the Plancherel theorem (\ref{Plancherel}), we therefore see that for every $t\in \mathbb F_q,$
$$\nu(t)\lesssim \frac{|E||F|}{q}+q^d \sqrt{|E||F|}  \left( \max_{m\neq (0,\dots,0)} |\widehat{V_t}(m)|\right).$$
From our hypotheses (\ref{hypothesis1}), (\ref{hypothesis2}), it follows that
$$ \nu(t)\lesssim \frac{|E||F|}{q}+q^{\frac{d-1}{2}} \sqrt{|E||F|} \quad \mbox{if}~~t\notin A$$ and
$$\nu(t)\lesssim \frac{|E||F|}{q}+q^{\frac{d}{2}} \sqrt{|E||F|}\quad \mbox{if} ~~t\in A.$$
By these inequalities and  the definition of the counting function $\nu(t)$, we see that
$$ |E||F|= \sum_{t\in \Delta_P(E,F)} \nu(t) = \sum_{t\in A\cap \Delta_P(E,F)}\nu(t) + \sum_{t\in (\mathbb F_q \setminus A)\cap \Delta_P(E,F)}\nu(t)$$
$$ \lesssim \frac{|E||F|}{q}+q^{\frac{d}{2}} \sqrt{|E||F|} + \left(\frac{|E||F|}{q}+q^{\frac{d-1}{2}} \sqrt{|E||F|}\right) | \Delta_P(E,F)|,$$
where we used the fact that $|A|\sim 1.$
Note that  if  $|E||F|\geq C q^{d}$ for some $C>0$ sufficiently large, then $ |E||F|\sim |E||F| + \frac{|E||F|}{q}+q^{\frac{d}{2}} \sqrt{|E||F|}.$
From this fact and above estimate, we conclude that if  $|E||F|\geq C q^{d}$ for some $C>0$ sufficiently large, then
$$ | \Delta_P(E,F)|\gtrsim \frac{|E||F|} { \frac{|E||F|}{q}+q^{\frac{d-1}{2}} \sqrt{|E||F|}}$$
which completes the proof.
\end{proof}
\begin{remark}\label{Koh1} From the proof of Theorem \ref{Erdosformula}, it is clear that if $A$ is an empty set, then we can drop the assumption that $|E||F|\geq C q^{d}$ for some $C>0$ sufficiently large. As an example showing that $A$ can be an empty set, Koh \cite{Ko09} showed that if the dimension $d\geq3$ is odd and $P(x)=\sum_{j=1}^d a_jx_j^2$ with $a_j\neq 0$, then $|\widehat{V_t}(m)|\lesssim q^{-(d+1)/2}$ for all $m\neq (0,\dots,0), t\in \mathbb F_q.$
\end{remark}
Combining Theorem \ref{Erdosformula} with Lemma \ref{Todd}, the following corollary  immediately follows.
\begin{corollary}\label{Corkoh} Let $P(x)=\sum_{j=1}^d a_j x_j^s \in \mathbb F_q[x_1,\dots,x_d]$ for $s\geq 2$ integer and $a_j\neq 0.$
Assume that the characteristic of $\mathbb F_q$ is sufficiently large.
If $E,F\subset \mathbb F_q^d$ with  $|E||F|\geq C q^{d}$ for some $C>0$ sufficiently large, then we have
$$ |\Delta_P(E,F)| \gtrsim \min\left( q, q^{-\frac{(d-1)}{2}} \sqrt{|E||F|}\right).$$
\end{corollary}
As pointed out in Remark \ref{Koh1}, if $s=2$ and $d$ is odd, then the conclusion in Corollary \ref{Corkoh} holds without the assumption that $|E||F|\gtrsim q^{d}.$
\section{Simple formula for generalized Falconer distance problems }
In previous section, we have seen that the distance problems are closely related to decays of the Fourier transforms on varieties.
In order to apply Theorem \ref{alldistanceformula} or Theorem \ref{Erdosformula}, we must estimate the Fourier decay of the variety $V_t=\{x\in \mathbb F_q^d: P(x)=t\}.$ In general, it is not easy to estimate the Fourier transform of $V_t.$
To do this, we need to show the following exponential sum estimate holds: for $m\in \mathbb F_q^d\setminus \{(0,\dots,0)\},$
$$ \widehat{V_t}(m)=q^{-d}\sum_{x\in V_t} \chi(-x\cdot m)= q^{-d-1}\sum_{(x,s)\in \mathbb F_q^{d+1}} \chi( sP(x)-m\cdot x-st) \lesssim q^{-\frac{d+1}{2}},$$
where the second equality follows from the orthogonality relation of the canonical additive character $\chi.$
In other words, we must show that for $m\neq (0,\dots,0),$
\begin{equation}\label{hard}\sum_{(x,s)\in \mathbb F_q^{d+1}} \chi( sP(x)-m\cdot x-st) \lesssim q^\frac{d+1}{2}.\end{equation}

Can we find a more useful, easier formula for distance problems than the formulas given in Theorem \ref{alldistanceformula} or Theorem \ref{Erdosformula}? If we are just interested in getting the positive proportion of all distances, then the answer is yes.
We do not need to estimate the size of $V_t$ and we just need to estimate more simple exponential sums.
We have the following simple formula.
\begin{theorem}\label{easyformula} Let $P(x)\in \mathbb F_q[x_1,\dots,x_d]$ be a polynomial.
Given $E,F\subset \mathbb F_q^d,$ define the distance set
$$ \Delta_P(E,F)=\{P(x-y)\in \mathbb F_q: x\in E, y\in F\}.$$
Suppose that  the following estimate holds: for every $m\in \mathbb F_q^d$ and $s\neq 0,$
\begin{equation}\label{easy1}\left| \sum_{x\in \mathbb F_q^d} \chi(s P(x)+m \cdot x )\right| \lesssim q^{\frac{d}{2}} .\end{equation}
Then, if $ |E||F|\geq C q^{d+1}$ with $C>0$ sufficiently large, then $ |\Delta_P(E,F)|\gtrsim q.$
\end{theorem}

Notice that the estimate (\ref{easy1}) is easier than the estimate (\ref{hard}). We shall see that Theorem \ref{easyformula} can be obtained by studying the distance problem related to the generalized paraboloid in $\mathbb F_q^{d+1}.$ The details and the proof of Theorem \ref{easyformula} will be given in the next subsections.
Using Theorem \ref{easyformula}, we have the following corollary.
\begin{corollary}\label{CC1} Let $P(x)=\sum_{j=1}^d a_j x_j^{c_j}$ for $c_j\geq 2$ integers, $a_j\neq 0,$ and gcd$(c_j, q)=1$ for all $j.$
Let $E, F\subset \mathbb F_q^d.$ Define $\Delta_P(E,F)=\{P(x-y)\in \mathbb F_q: x\in E, y\in F\}.$
If $|E||F|\geq C q^{d+1}$ with $C>0$ sufficiently large, then $|\Delta_P(E,F)|\gtrsim q.$
\end{corollary}
\begin{proof} From Theorem \ref{easyformula}, it suffices to show that the estimate (\ref{easy1}) holds.
However, this is an immediate result from Weil's theorem (Theorem \ref{Weil}) and the proof is complete.
\end{proof}

\begin{remark} We stress that Corollary \ref{Cor1} does not imply  Corollary \ref{CC1} above.
Considering the diagonal polynomial $P(x)=\sum_{j=1}^d a_j x_j^{c_j}$, if the exponents $c_j$ are distinct, then Corollary \ref{Cor1} does not give any information. Authors in this paper have not found any reference which shows that for $m\in \mathbb F_q^d\setminus \{(0,\dots,0)\},$ and $t\neq 0,$
$$ |\widehat{V_t}(m)|\lesssim q^{-\frac{d+1}{2}}, $$
where $V_t=\{x\in \mathbb F_q^d: \sum_{j=1}^d a_j x_j^{c_j}=t\}$ and all $c_j$ are not same. Thus, we can not apply Theorem \ref{alldistanceformula} to obtain such result as in Corollary \ref{CC1}. In conclusion, Theorem \ref{easyformula} can be very powerful to study the generalized Falconer distance problems.
We remark that using some powerful results from algebraic geometry we can find more concrete examples of polynomials satisfying (\ref{easy1}) or (\ref{hard}). For example, see Theorem 8.4 in \cite{De73} or Theorem 9.2 in \cite{DL91}.

\end{remark}

\subsection{Distance problems related to generalized paraboloids}
In this subsection, we shall find a useful theorem which yields the simple distance formula in Theorem \ref{easyformula}.
Let $E, F\subset {\mathbb F_q^d}$ are product sets. In the case when  $E=F$ and $P(x)=x_1^2+\cdots+x_d^2,$ it is well known in \cite{CEHIK09} that if $|E||F|\gtrsim q^{2d^2/(2d-1)},$ then $|\Delta_P(E,F)|\gtrsim q.$ This improves the Falconer exponent $(d+1)/2.$ Here, we also study the generalized Falconer distance problems for product sets, related to the generalized paraboloid distances which are different from the usual spherical distance. If a distance set is related to usual spheres or paraboloids, then we can take advantage of the explicit forms in the varieties.
In these settings, if $E$ and $F$ are product sets in ${\mathbb F_q^d},$ we may easily get the improved Falconer distance result, $|E||F|\gtrsim q^{2d^2/(2d-1)}.$ However, the polynomial generating a distance set is not given in an explicit form, then the generalized distance problem can be hard. We are interested in getting the improved Falconer result on the generalized distance problems for product sets,
associated with generalized paraboloids defined as in below. Moreover, we aim to apply the result to proving Theorem \ref{easyformula}.
To achieve our aim, we shall work on $\mathbb F_q^{d+1}$ in stead of $\mathbb F_q^d, d\geq 1.$ We now introduce the generalized paraboloid in $\mathbb F_q^{d+1}.$ Given a polynomial $P(x)\in \mathbb F_q[x_1,\dots,x_d]$ and $t\in \mathbb F_q$, we define the generalized paraboloid $V_t\subset \mathbb F_q^{d+1}$ as the set
$$V_t=\{(x, x_{d+1})\in {\mathbb F_q^d}\times {\mathbb F_q}: P(x)-x_{d+1}=t\},$$
It is clear that $|V_t|=q^d$ for all $t\in \mathbb F_q,$ because if we fix $x \in \mathbb F_q^d$, then $x_{d+1}$ is uniquely determined. If the polynomial is given by
$P(x)=x_1^2+\dots+x_d^2$, then $V_0$ is exactly the usual paraboloid in $\mathbb F_q^{d+1}.$ Let $H(x,x_{d+1})=P(x)-x_{d+1},$ where $H$ is a polynomial in $\mathbb F_q[x_1,\dots,x_d, x_{d+1}].$ Given $E^*, F^* \subset \mathbb F_q^{d+1}$ and $P(x)\in \mathbb F_q[x_1,\dots,x_d],$ consider the generalized distance set
$$\Delta_H(E^*,F^*)=\{ H(x-y,x_{d+1}-y_{d+1})\in \mathbb F_q: (x,x_{d+1}) \in E^*, (y,y_{d+1})\in F^*\},$$
where $H(x,x_{d+1})=P(x)-x_{d+1}.$
One may have the following question.
What kinds of conditions on the polynomial $P(x) \in \mathbb F_q[x_1,\dots,x_{d}]$ do we need to get the improved Falconer exponent for the distance problems associated with the product sets $E^*$ and $F^*$ in $\mathbb F_q^{d+1}$?
The following theorem may answer for above question.
\begin{theorem}\label{excellent}
Let $P(x)\in \mathbb F_q[x_1,\dots,x_d]$ be a polynomial with degree $\geq 2$ satisfying the following condition:
for each $s\neq 0$ and $m \in \mathbb F_q^d,$
\begin{equation}\label{keyassumption}\left| \sum_{x \in {\mathbb F_q^{d}}} \chi(s P(x)+m \cdot x )\right| \lesssim q^{\frac{d}{2}}.\end{equation}
If $E^*=E \times E_{d+1}$ and $ F^*=F \times F_{d+1}$ are product sets in $\mathbb F_q^{d}\times \mathbb F_q$, and
$\frac{|E^*||F^*|}{|F_{d+1}|} \geq Cq^{d+1}$ with $C>0$ sufficiently large, then we have
$$ |\Delta_H(E^*,F^*)|= |\{ H(x-y,x_{d+1}-y_{d+1})\in \mathbb F_q: (x,x_{d+1}) \in E^*, (y,y_{d+1})\in F^*\}|\gtrsim q,$$
where $H(x,x_{d+1})=P(x)-x_{d+1}.$
\end{theorem}
\begin{proof}
Let $E^*, F^*\subset {\mathbb F_q^{d+1}}$ be product sets given by the forms:
$E^*=E \times E_{d+1}$ and  $ F^*=F \times F_{d+1}$ in $\mathbb F_q^{d}\times \mathbb F_q.$
In addition, assume that $\frac{|E^*||F^*|}{|F_{d+1}|} \gtrsim q^{d+1}.$
Let $x^*, y^* \in \mathbb F_q^{d+1}.$
As before, consider the counting function $\nu$ on ${\mathbb F_q}$ given by
$$ \nu(t)=\left| \{(x^*,y^*)\in E^*\times F^*: H(x^*-y^*)=t\}\right|,$$
For each $t\in \mathbb F_q,$ let
$$V_t=\{x^*\in \mathbb F_q^{d+1}: H(x^*)-t=0\}.$$

We are interested in measuring the lower bound of the distance set $ \Delta_H(E^*,F^*)$ defined by
$$ \Delta_H(E^*,F^*)=\{ H(x^*-y^*)\in {\mathbb F_q}: x^*\in E^*, y^*\in F^*\}.$$
In $(d+1)$ dimension, applying the Fourier inversion theorem (\ref{Finversion}) to the function $V_t(x^*-y^*)$ and using the definition of the Fourier transforms (\ref{Ftransform}), we have
\begin{align*}\nu(t)=& \sum_{x^*\in E^*, y^*\in F^*} V_t(x^*-y^*)\\
=& q^{2(d+1)} \sum_{m^*\in \mathbb F_q^{d+1}}\overline{\widehat{E^*}}(m^*) \widehat{F^*}(m^*) \widehat{V_t}(m^*)\\
=&q^{2(d+1)} \overline{\widehat{E^*}}(0,\dots,0) \widehat{F^*}(0,\dots,0) \widehat{V_t}(0,\dots,0) +
q^{2(d+1)} \sum_{m^*\in {\mathbb F_q^{d+1}}\setminus \{(0,\dots,0)\}}\overline{\widehat{E^*}}(m^*) \widehat{F^*}(m^*) \widehat{V_t}(m^*)\\
=&\frac{|E^*||F^*|}{q} + q^{2(d+1)} \sum_{m^*\in {\mathbb F_q^{d+1}}\setminus \{(0,\dots,0)\}}\overline{\widehat{E^*}}(m^*) \widehat{F^*}(m^*) \widehat{V_t}(m^*).
\end{align*}
Squaring the $\nu(t)$ and summing it over $t\in \mathbb F_q$ yield that
\begin{align*} \sum_{t\in {\mathbb F_q}} \nu^2(t)=& \frac{|E^*|^2|F^*|^2}{q} + 2q^{2d+1} |E^*||F^*|
\sum_{m^* \in {\mathbb F_q^{d+1}}\setminus \{(0,\dots,0)\}}\overline{\widehat{E^*}}(m^*) \widehat{F^*}(m^*) \sum_{t\in {\mathbb F_q}}\widehat{V_t}(m^*)\\
&+q^{4(d+1)} \sum_{m^*,\xi^* \in {\mathbb F_q^{d+1}}\setminus \{(0,\dots,0)\}}
\overline{\widehat{E^*}}(m^*) \widehat{F^*}(m^*) \overline{\widehat{E^*}}(\xi^*) \widehat{F^*}(\xi^*)
 \sum_{t\in {\mathbb F_q}}\widehat{V_t}(m^*)\widehat{V_t}(\xi^*)\\
 =&\mbox{I} + \mbox{II}+\mbox{III}.\end{align*}
 Observe that $\mbox{I}$ and $\mbox{II}$ are given by
 \begin{equation}\label{estimateIandII}
 \mbox{I}= \frac{|E^*|^2|F^*|^2}{q} \quad \mbox{and} ~~\mbox{II}=0,\end{equation}
 where $\mbox{II}=0$ follows immediately from the fact that $ \sum_{t\in {\mathbb F_q}}\widehat{V_t}(m^*) =0$ for $m^*\neq (0,\dots,0).$ In order to estimate $\mbox{III}$, first observe that for $m^*=(m,m_{d+1})\in\mathbb F_q^{d+1},$
 $$ \widehat{V_t}(m^*)=\frac{1}{q^{d+1}} \sum_{x\in {\mathbb F_q^{d}}} \chi(-m_{d+1} P(x) -m \cdot x) \chi(tm_{d+1}).$$
It therefore follows that for $m^*=(m,m_{d+1}), \xi^*=(\xi, \xi_{d+1} )\in\mathbb F_q^{d+1},$
$$ \widehat{V_t}(m^*) \widehat{V_t}(\xi^*) =\frac{1}{q^{2(d+1)}}\sum_{x, y \in {\mathbb F_q^{d}}}\chi(t(m_{d+1}+\xi_{d+1}))
\chi(-m_{d+1}P(x)-m \cdot x) \chi(-\xi_{d+1} P(y)-\xi \cdot y).$$
Notice that if $m^*\neq (0,\dots,0)$ and $m_{d+1}=0$, then $\widehat{V_t}(m^*)$ vanishes. In addition, observe that if $m_{d+1}+\xi_{d+1} \neq 0,$ then
$\sum_{t\in {\mathbb F_q}} \widehat{V_t}(m^*) \widehat{V_t}(\xi^*)$ also vanishes and if $m_{d+1}+\xi_{d+1} =0,$ then
$\sum_{t\in {\mathbb F_q}}\chi(t (m_{d+1}+\xi_{d+1}))=q$. From these observations together with a change of a variable, $m_{d+1}\to s$, we obtain that
$$\mbox{III}= q^{2d+3} \sum_{m, \xi \in {\mathbb F_q^{d}}}\sum_{s\in {\mathbb F_q}\setminus \{0\}}
\overline{\widehat{E^*}}(m, s) \widehat{F^*}(m,s) \overline{\widehat{E^*}}(\xi, -s) \widehat{F^*}(\xi, -s)
W(m, \xi, s, P),$$
where $W(m, \xi, s, P)=\sum_{x, y \in {\mathbb F_q^{d}}} \chi(-sP(x)-m\cdot x)\chi(sP(y)-\xi\cdot y).$
Our assumption (\ref{keyassumption}) implies that for each $s\neq 0$ and $m , \xi \in \mathbb F_q^{d},$
\begin{equation}\label{assumption}\left|W(m, \xi, s, P)\right|=\left|\sum_{x, y \in {\mathbb F_q^{d}}} \chi(-sP(x)-m\cdot x)\chi(sP(y)-\xi\cdot y)\right|\lesssim q^{d}.\end{equation}
Since $ E^*=E \times E_{d+1} \quad \mbox{and}~~ F^*=F \times F_{d+1}$,  it is clear that
$$\widehat{E^*}(m, s)= \widehat{E}(m) \widehat{E_{d+1}}(s)\quad \mbox{and}~~ \widehat{F^*}(m, s)= \widehat{F}(m) \widehat{F_{d+1}}(s).$$

  Using this fact along with the inequality (\ref{assumption}), we see that
$$ |\mbox{III}|\lesssim q^{3(d+1)} \left(\sum_{m \in\mathbb F_q^d} \left|\widehat{E}(m) \widehat{F}(m)\right|\right)^2 \left( \sum_{s\in {\mathbb F_q}\setminus \{0\}} \left|\widehat{E_{d+1}}(s) \widehat{F_{d+1}}(s)\right|^2\right).$$
Using the Cauchy-Schwarz inequality and the trivial bound $ |\widehat{F_{d+1}}(s)|\leq |\widehat{F_{d+1}}(0)|= \frac{|F_{d+1}|}{q},$
we obtain that
$$ |\mbox{III}|\lesssim  q^{3d+1} |F_{d+1}|^2 \left( \sum_{m \in {\mathbb F_q^{d}}} \left|\widehat{E}(m)\right|^2\right)\left( \sum_{m \in {\mathbb F_q^{d}}} \left|\widehat{F}(m)\right|^2\right) \left( \sum_{s\in {\mathbb F_q}} \left|\widehat{E_{d+1}}(s)\right|^2\right).$$

Using the Plancherel theorem (\ref{Plancherel}) yields the following:
\begin{equation}\label{finalestimateIII}
|\mbox{III}|\lesssim q^{d} |E||E_{d+1}||F| |F_{d+1}|^2 = q^{d} |E^*||F^*||F_{d+1}|.
\end{equation}

Putting estimates (\ref{estimateIandII}), (\ref{finalestimateIII}) together, we conclude that
$$ \sum_{t\in {\mathbb F_q}} \nu^2(t) \lesssim \frac{|E^*|^2|F^*|^2}{q}+ q^{d} |E^*||F^*||F_{d+1}|.$$
By the Cauchy-Schwarz inequality, we  see that
\begin{align*} |E^*|^2|F^*|^2 &= \left( \sum_{t\in \Delta_H(E^*,F^*)} \nu(t) \right)^2 \\
&\leq |\Delta_H(E^*,F^*)| \left(\sum_{t\in {\mathbb F_q}} \nu^2(t)\right) \end{align*}

Thus, we have proved the following:
$$ |\Delta_H(E^*,F^*)|\gtrsim \min\left(q, q^{-d}|E^*||F^*||F_{d+1}|^{-1} \right).$$
This implies that if $\frac{|E^*||F^*|}{|F_{d+1}|}\gtrsim q^{d+1},$ then
$$ |\Delta_H(E^*,F^*)|\gtrsim q,$$
which completes the proof.

\end{proof}

\subsection{Proof of Theorem \ref{easyformula}}
We prove that the general paraboloid distance problem for product sets in $\mathbb F_q^{d+1}$ implies the generalized distance problem in $\mathbb F_q^d.$ Namely, Theorem \ref{easyformula} can be obtained as a corollary of Theorem \ref{excellent}.
In order to prove Theorem \ref{easyformula}, first  fix $E,F\subset \mathbb F_q^d$ with $|E||F|\geq C q^{d+1}$ with $C>0$ large. Let $E^*=E\times \{0\} \subset \mathbb F_q^{d+1}$ and $F^*=F\times \{0\}\subset \mathbb F_q^{d+1}.$ Observe that $|E|=|E^*|, |F|=|F^*|,$ and
$$ |\Delta_P(E,F)|=|\{P(x-y)\in \mathbb F_q: x\in E, y\in F\}|$$
$$=|\Delta_H(E^*,F^*)|=|\{H(x-y, x_{d+1}-y_{d+1}) \in \mathbb F_q: (x,x_{d+1})\in E^*, (y,y_{d+1}) \in F^*\}|$$
where $H(x,x_{d+1})= P(x)-x_{d+1}.$
The assumption (\ref{easy1}) in Theorem \ref{easyformula} implies that  the conclusion of
 Theorem \ref{excellent} holds: if
$\frac{|E^*||F^*|}{|\{0\}|} \gtrsim q^{d+1},$ then $|\Delta_H(E^*,F^*)|\gtrsim q.$
Since $|\{0\}|=1, |E^*|=|E|, |F^*|=|F|,$ and $ |\Delta_H(E^*,F^*)|=|\Delta_P(E,F)|,$ we therefore conclude that
if $ |E||F|\gtrsim q^{d+1},$ then $ |\Delta_P(E,F)|\gtrsim q.$ Thus, the proof of Theorem \ref{easyformula} is complete.

\section{Generalized pinned distance problems}

We find the conditions on the polynomial $P(x)\in \mathbb F_q[x_1,\dots,x_d]$ such that the desirable results for generalized pinned distance problems hold.
First, let us introduce some notation associated with the pinned distance problems.
Let $P(x)\in \mathbb F_q[x_1,\dots,x_d]$ be a polynomial. For each $t\in \mathbb F_q$, we define a variety $V_t$ by
$$ V_t=\{x\in \mathbb F_q^d: P(x)=t\}.$$
The Schwartz-Zippel Lemma (Lemma \ref{SZlemma}) says that $|V_t|\lesssim q^{d-1}$ for all $t\in \mathbb F_q.$
Let $E\subset \mathbb F_q^d.$ Given $y\in \mathbb F_q^d, $ we denote by $\Delta_P(E,y)$ a pinned distance set defined as
$$ \Delta_P(E,y)=\{P(x-y)\in \mathbb F_q : x\in E\}.$$
We are interested in finding the element $y\in \mathbb F_q^d$ and the size of $E\subset \mathbb F_q^d$ such that
$|\Delta_P(E,y)|\gtrsim q.$ We have the following theorem.
\begin{theorem}\label{mainpintheorem}
Let $T\subset \mathbb F_q$ with $|T|\sim 1.$
Suppose that the varieties $V_t,$ generated by a polynomial $P(x) \in \mathbb F_q[x_1,\cdots,x_d],$ satisfy the following:
for all $m\in \mathbb F_q^d\setminus \{(0,\cdots,0)\},$
\begin{equation}\label{ju1} |\widehat{V_t}(m)|\lesssim q^{-\frac{d+1}{2}}\quad \mbox{if} ~~t\notin T\end{equation}
and
\begin{equation}\label{ju2} |\widehat{V_t}(m)|\lesssim q^{-\frac{d}{2}}\quad \mbox{if}~~ t\in T.\end{equation}
Let $E, F\subset \mathbb F_q^d.$  If $|E||F|\geq C q^{d+1}$ with $C>0$ large enough, then there exists  $F_0 \subset F$ with $|F_0|\sim |F|$ such that
$$|\Delta_P(E,y)|\gtrsim q \quad\mbox{for all} ~~y\in F_0.$$
\end{theorem}

\begin{proof} Using the pigeonhole principle, it suffices to prove that if $ |E||F|\gtrsim q^{d+1},$ then

\begin{equation}\label{aimpin} \frac{1}{|F|}\sum_{y\in F} |\Delta_P(E,y)| \gtrsim q.\end{equation}
For each $t\in \mathbb F_q$ and $y\in F$, consider the counting function $\nu_y(t)$ given by
$$ \nu_y(t)=|\{x\in E: P(x-y)=t\}|=|\{x\in E: x-y\in V_t\}|.$$
Applying the Fourier inversion transform  to the function $V_t(x-y)$ and using the definition of the Fourier transform, we see that
\begin{align*}
\nu_y(t)&=\sum_{x\in \mathbb F_q^d} E(x) V_t(x-y)=q^d\sum_{m\in \mathbb F_q^d}\overline{\widehat{E}}(m) \widehat{V_t}(m)\chi(-m\cdot y)\\
&= q^d \overline{\widehat{E}}(0,\dots,0) \widehat{V_t}(0,\dots,0)\chi(0)
+ q^d\sum_{m\in \mathbb F_q^d\setminus \{(0,\dots,0)\}}\overline{\widehat{E}}(m) \widehat{V_t}(m)\chi(-m\cdot y)\\
&= \frac{|E||V_t|}{q^d} +q^d\sum_{m\in \mathbb F_q^d\setminus \{(0,\dots,0)\}}\overline{\widehat{E}}(m) \widehat{V_t}(m)\chi(-m\cdot y).
\end{align*}

Squaring the $\nu_y(t)$ and summing it over $y\in F$ and $t\in \mathbb F_q$, we see that
\begin{align*} \sum_{y\in F}\sum_{t\in \mathbb F_q} \nu^2_y(t)
&=\sum_{y\in F}\sum_{t\in \mathbb F_q} \frac{|E|^2|V_t|^2}{q^{2d}} \\
&+\sum_{y\in F}\sum_{t\in \mathbb F_q} 2 |E||V_t| \sum_{m\in \mathbb F_q^d\setminus \{(0,\dots,0)\}}\overline{\widehat{E}}(m)\widehat{V_t}(m)\chi(-m\cdot y)\\
&+\sum_{y\in F}\sum_{t\in \mathbb F_q} q^{2d} \sum_{m,\xi \in \mathbb F_q^d\setminus \{(0,\dots,0)\}}
\overline{\widehat{E}}(m)\widehat{V_t}(m)\chi(-m\cdot y)\overline{\widehat{E}}(\xi)\widehat{V_t}(\xi)\chi(-\xi\cdot y)\\
&=\mbox{A} + \mbox{B}+\mbox{C}.
\end{align*}
Since $|V_t|\lesssim q^{d-1}$ for all $t\in \mathbb F_q$, it is clear that
\begin{equation}\label{termA}
|\mbox{A}| \lesssim \frac{|E|^2|F|}{q}.\end{equation}
To estimate $|\mbox{B}|$, first use the definition of the Fourier transform and find the maximum value of the sum in $t\in \mathbb F_q$ with respect to $m \in \mathbb F_q^d\setminus \{0,\dots,0)\}$. Namely, we have
$$|\mbox{B}| \leq 2 q^d |E| \left( \max_{m\in \mathbb F_q^d\setminus \{(0,\dots,0)\} } \sum_{t\in \mathbb F_q} |V_t||\widehat{V_t}(m)| \right) \sum_{m\in \mathbb F_q^d \setminus \{(0,\dots,0)\}} |\overline{\widehat{E}}(m)||\widehat{F}(m)|.$$
From the assumptions, (\ref{ju1}), (\ref{ju2}), $|T|\sim 1,$ and the fact that $|V_t|\lesssim q^{d-1} $ for all $t\in \mathbb F_q$, we see
that the maximum value term is $\lesssim q^{(d-1)/2}.$ If we use the Cauchy-Schwarz inequality and the Plancherel theorem, then we also see that
$$\sum_{m\in \mathbb F_q^d \setminus \{(0,\dots,0)\}} |\overline{\widehat{E}}(m)||\widehat{F}(m)|
\leq \frac{|E|^{\frac{1}{2}}  |F|^{\frac{1}{2}}}{q^d} .$$
Therefore, the value \mbox{B} can be estimated by
\begin{equation}\label{termB}
|\mbox{B}|\lesssim q^{\frac{d-1}{2}} |E|^{\frac{3}{2}} |F|^{\frac{1}{2}}.\end{equation}
Now we estimate the value $\mbox{C}.$ Using a change of the variable, $\xi\to -\xi$, we see
$$ \sum_{m,\xi \in \mathbb F_q^d\setminus \{(0,\dots,0)\}}
\overline{\widehat{E}}(m)\widehat{V_t}(m)\chi(-m\cdot y)\overline{\widehat{E}}(\xi)\widehat{V_t}(\xi)\chi(-\xi\cdot y)$$
$$= \left| \sum_{m \in \mathbb F_q^d\setminus \{(0,\dots,0)\}}
\overline{\widehat{E}}(m)\widehat{V_t}(m)\chi(-m\cdot y)\right|^2$$
which is always a nonnegative real number. In order to obtain an upper bound of the term $\mbox{C}$, we therefore expand the sum over $y\in F$ to the sum over $y\in \mathbb F_q^d$  and we compute the sum in $y$ by using the orthogonality relation of the canonical additive character $\chi.$ It therefore follows that
\begin{align*} |C|&\leq q^{3d}  \sum_{m\in \mathbb F_q^d\setminus \{(0,\dots,0)\}} \sum_{t\in \mathbb F_q}|\widehat{V_t}(m)|^2
|\widehat{E}(m)|^2\\
&\leq q^{3d} \left(\max_{m\in \mathbb F_q^d\setminus \{(0,\dots,0)\}} \sum_{t\in \mathbb F_q}|\widehat{V_t}(m)|^2     \right)
 \sum_{m\in \mathbb F_q^d} |\widehat{E}(m)|^2.
\end{align*}
Using the Plancherel theorem and the assumption of the Fourier decay of $V_t$, we see that
$$ \sum_{m\in \mathbb F_q^d} |\widehat{E}(m)|^2 = \frac{|E|}{q^d} \quad \mbox{and}~~
\max_{m\in \mathbb F_q^d\setminus \{(0,\dots,0)\}} \sum_{t\in \mathbb F_q}  |\widehat{V_t}(m)|^2  \lesssim q^{-d}.$$
Putting these facts together yields the upper bound of the value \mbox{$|C|$}:
\begin{equation}\label{termC}
|\mbox{C}|\lesssim q^d |E|.\end{equation}
From (\ref{termA}), (\ref{termB}), and (\ref{termC}), we obtain the following estimate:
$$\sum_{t\in F}\sum_{t\in \mathbb F_q} \nu^2_y(t) \lesssim \frac{|E|^2|F|}{q}+q^{\frac{d-1}{2}} |E|^{\frac{3}{2}} |F|^{\frac{1}{2}}+q^d |E|.$$
Observe that if $|E||F|\geq C q^{d+1}$ for $C>0$ sufficiently large, then
\begin{equation}\label{keysum}
\sum_{y\in F}\sum_{t\in \mathbb F_q} \nu^2_y(t) \lesssim \frac{|E|^2|F|}{q}.
\end{equation}
We are ready to finish the proof.
For each $y\in F$, if we note that $\sum_{t\in \Delta_P(E,y)} \nu_y(t) = |E|$ and then apply the Cauchy-Schwarz inequality , then we see
$$ |E|^2|F|^2 =\left( \sum_{y\in F}\sum_{t\in \Delta_P(E,y)} \nu_y(t) \right)^2 \leq \left(\sum_{y\in F} |\Delta_P(E,y)|\right) \left( \sum_{y\in F} \sum_{t\in \mathbb F_q} \nu^2_y(t)\right).$$
$$  \lesssim \left(\sum_{y\in F} |\Delta_P(E,y)|\right)  \frac{|E|^2|F|}{q}         $$
where the last line follows from the estimate (\ref{keysum}).
Thus, the estimate (\ref{aimpin}) holds and we complete the proof of Theorem \ref{mainpintheorem}.
\end{proof}

\begin{remark} Let $E,F\subset \mathbb F_q^d.$
We note that if $P(x_1,..,x_d)=a_1x_1^{s}+\cdots+a_dx_d^{s}$  satisfies the assumptions in Corollary \ref{Cor1}, then there exists a subset $F_0$ of $F$ with $|F_0|\sim |F|$ such that
$$|\Delta_P(E,y)|\gtrsim q \quad\mbox{for all} ~~y\in F_0.$$ This is an immediate result from Theorem \ref{mainpintheorem} and Lemma \ref{Todd}.
In terms of the generalized Falconer distance problem, this result sharpens the statement of Corollary \ref{Cor1}. On the other hand, Corollary \ref{Cor1} gives us the exact number of the elements in the distance set. \end{remark}

We close this paper by introducing a corollary of Theorem \ref{mainpintheorem}, which sharpens and generalizes the Vu's result (\ref{byvu}).
\begin{corollary} Let $P(x)\in \mathbb F_q[x_1,x_2] $ be a non-degenerate polynomial. If $|E||F|\geq C q^{3}$ for $E,F\subset \mathbb F_q^2$ and $C>0$ sufficiently large, then there exists a subset $F_0$ of $F$ with $|F_0|\sim |F|$ such that
$$|\Delta_P(E,y)|\gtrsim q \quad\mbox{for all} ~~y\in F_0.$$
\end{corollary}
\begin{proof} The proof follows immediately by applying Theorem \ref{mainpintheorem} along with Theorem \ref{Vuthm} and (\ref{Jfinal}) in Remark \ref{doit}.
\end{proof}

\end{document}